\documentclass[12pt,bookmark=true]{article}

 \RequirePackage{geometry}
\usepackage{graphicx}				

\usepackage{amssymb}
\usepackage{amsmath}
\usepackage{amsthm}
\usepackage{enumerate}
 \usepackage{lineno}

\newtheorem{theorem}{\indent Theorem}[section]

\newtheorem{proposition}[theorem]{\indent Proposition}
\newtheorem{definition}[theorem]{\indent Definition}
\newtheorem{lemma}{\indent Lemma}[section]
\newtheorem{remark}[theorem]{\indent Remark}
 
\begin{document}
\title{Global Carleman estimates for the fourth order parabolic equations and application to null controllability}
\author{
Bo You\footnote{Email address: youb2013@xjtu.edu.cn}
 \\
{\small School of Mathematics and Statistics, Xi'an Jiaotong University} \\
\small Xi'an, 710049, P. R. China\\
Fang Li\footnote{Email address: fli@xidian.edu.cn}
	 \\
	{\small School of Mathematics and Statistics, Xidian University} \\
	\small Xi'an, 710071, P. R. China
}

\maketitle

\begin{center}
\begin{abstract}
The main objective of this paper is to establish the null controllability for the fourth order semilinear parabolic equations with the nonlinearities involving the state and its gradient up to second order. First of all, based on optimal control theory of partial differential equations and global Carleman estimates obtained in \cite{gs} for fourth order parabolic equation with $L^2(Q)$-external force, we establish the global Carleman estimates for the $L^2(Q)$ weak solutions of the same system with a low regularity external term and some linear terms including the derivatives of the state up to second order.  Then we prove the null controllability of the fourth order semilinear parabolic equations by such global Carleman estimates and the Leray-Schauder's fixed points theorem.


\end{abstract}
\end{center}

\section{Introduction}
\def\theequation{1.\arabic{equation}}\makeatother
\setcounter{equation}{0}
Let $\mathcal{O}\subset \mathbb{R}^n (n\geq 2)$ be a nonempty bounded connected open set with smooth boundary $\partial\mathcal{O},$ $T>0$ and $\omega\subset\mathcal{O}$ is a small nonempty open subset  which is usually referred to as a control domain. Denote by $Q=\mathcal{O}\times (0,T),$ $\Sigma=\partial\mathcal{O}\times (0,T),$ $Q_\omega=\omega\times(0,T).$

In this paper, we mainly consider the null controllability for the fourth order semilinear parabolic equation:
\begin{equation}\label{eq1.1}
\begin{cases}
Ly=\frac{\partial y}{\partial t}+\Delta^2 y+a_0y+B_0\cdot\nabla y+D:\nabla^2y+a_1\Delta y=F(y,\nabla y, \nabla^2y)+v\chi_\omega+g,\,\,\,\,\forall\,\,\,(x,t)\in Q,\\
y=\Delta y=0,\,\,\,\,\,\forall\,\,\,\,(x,t)\in\Sigma,\\
y(x,0)=y_0(x),\,\,\,\,\forall\,\,\,\,x\in\mathcal{O}.
\end{cases}
\end{equation}
Here, the functions $a_0,$ $a_1\in L^{\infty}(Q;\mathbb{R}),$ $B_0=(B_{01},B_{02},\cdots,B_{0n})\in L^{\infty}(Q;\mathbb{R}^n),$ $D=(D_{ij})_{n\times n}\in L^{\infty}(Q;\mathbb{R}^{n^2}),$ $g\in L^2(Q)$ is a given externally applied force, the function $F:\mathbb{R}\times\mathbb{R}^2\times\mathbb{R}^{n^2}\rightarrow\mathbb{R}$ is locally Lipschitz continous, $\chi_\omega$ is the characteristic function of the set $\omega,$ $y_0$ is the given initial data and $v\in L^2(Q_\omega)$ is a control function to be determined. The unknown function $y$ may represent a scaled film height and the term $\Delta^2y$ represents the capillarity-driven surface diffusion.

In recent several years, there are some results about the controllability for fourth order parabolic equations in both one dimension (see \cite{cn5, cn7, cn6, ce1, ce, gp, hsv, lgm}) and the higher dimensions (see \cite{dji, gs, kk, lq, yh}). In particular, the approximate controllability and non-approximate controllability of higher order parabolic equations were studied in \cite{dji}. The author in \cite{yh} proved the null controllability of fourth order parabolic equations by using the ideas of \cite{lg}. It is worthy to mention that the Carleman inequality for a fourth order parabolic equation with $n\geq 2$ was first established in \cite{gs}. Later, the author in \cite{kk} proved the null controllability and the exact controllability to the trajectories at any time $T>0$ for the fourth order semi-linear parabolic equations with a control function acting at the interior.  Recently, the null controllability for fourth order stochastic parabolic equations was proved by duality arguments and a new global Carleman estimates in \cite{lq}. A unified weighted inequality for fourth-order partial differential operators was given in \cite {cy}. Moreover, they applied it to obtain the log-type stabilization result for the plate equation. As far as we know, the null controllability of problem \eqref{eq1.1} is equivalent to the observation of its corresponding adjoint problem in the case that $F\equiv 0.$ To develop the required observability inequality, it is very necessary to establish the global Carleman estimates of the adjoint system of problem \eqref{eq1.1} with $F\equiv 0$. However, there is no results concerning the desired global Carleman estimates and null controllability of problem \eqref{eq1.1}. Thus, we will establish the desired global Carleman estimates and prove the null controllability of problem \eqref{eq1.1}.

The rest of this paper is organized as follows: in Section 2, we will recall a well-established global Carleman estimates and give a definition of $L^2(Q)$-weak solution defined by transposition. In Section 3, we will establish the desired global Carleman inequality. Section 4 is devoted to the null controllability for the fourth order semilinear parabolic equation. 
\section{\bf Preliminaries}
\def\theequation{2.\arabic{equation}}\makeatother
\setcounter{equation}{0}
To formulate our Carleman inequality, we first introduce the following weight functions.
\begin{lemma}(\cite{fav})\label{le2.1}
Let $\omega_0\subset\omega$ be an arbitrary fixed subdomain of $\mathcal{O}$ such that $\overline{\omega_0}\subset\omega.$ Then there exists a function $\eta\in\mathcal{C}^4(\overline{\mathcal{O}})$ such that
\begin{align*}
\eta(x)>0,\,\,\,\,\forall\,\,\,x\in\mathcal{O};\,\,\,
\eta(x)=0,\,\,\,\,\forall\,\,\,x\in\partial\mathcal{O};\,\,\,
|\nabla\eta(x)|>0,\,\,\,\,\forall\,\,\,x\in\overline{\mathcal{O}\backslash\omega_0}.
\end{align*}
\end{lemma}
In order to state our Carleman inequality, we define some weight functions:
\begin{align}\label{eq2.1}
\alpha(x,t)=\frac{e^{\lambda(2\|\eta\|_{L^{\infty}(\mathcal{O})}+\eta(x))}-e^{4\lambda\|\eta\|_{L^{\infty}(\mathcal{O})}}}{\sqrt{t(T-t)}},\,\,\,\xi(x,t)=\frac{e^{\lambda(2\|\eta\|_{L^{\infty}(\mathcal{O})}+\eta(x))}}{\sqrt{t(T-t)}}.
\end{align}
Moreover, they possess the following properties:
\begin{align}\label{eq2.2}
\nabla\alpha=\nabla\xi=\lambda\xi\nabla\eta,\,\,
\xi^{-1}\leq \frac{T}{2},\,\,
|\alpha_t|+|\xi_t|\leq\frac{T}{2}\xi^3,\,\,\forall\,\,(x,t)\in Q.
\end{align}
In \cite{gs}, the global Carleman estimates for fourth order parabolic equations with $L^2(Q)$-source term has been established. Here, we state it as follows. 
\begin{lemma}\label{le2.2}(see \cite{gs})
Assume that $z_0\in L^2(\mathcal{O}),$ $g\in L^2(Q)$ and the function $\alpha,$ $\xi$ are defined by \eqref{eq2.1}. Then there exists $\hat{\lambda}>0$ such that for an arbitrary $\lambda\geq \hat{\lambda},$ we can choose $s_0=s_0(\lambda)>0$ satisfying: there exists a constant $C=C(\lambda)>0$ independent of $s,$ such that the solution $z\in L^2(Q)$ to problem
\begin{equation}\label{eq2.3}
\begin{cases}
L^*z=-\frac{\partial z}{\partial t}+\Delta^2 z=g,\,\,\,\,\forall\,\,\,(x,t)\in Q,\\
z=\Delta z=0,\,\,\,\,\,\forall\,\,\,\,(x,t)\in\Sigma,\\
z(x,T)=z_0(x),\,\,\,\,\forall\,\,\,\,x\in\mathcal{O}
\end{cases}
\end{equation}
satisfies the following inequality:
\begin{align*}
&\int_Qe^{2s\alpha}\left(s^6\lambda^8\xi^6|z|^2+s^4\lambda^6\xi^4|\nabla z|^2+s^3\lambda^4\xi^3|\Delta z|^2+s^2\lambda^4\xi^2|\nabla^2z|^2+s\lambda^2\xi|\nabla\Delta z|^2\right)\,dxdt\\
&+\int_Qe^{2s\alpha}\left(\frac{1}{s\xi}(|z_t|^2+|\Delta^2z|^2)\right)\,dxdt\\
&\leq C\left(\int_{Q_\omega}s^7\lambda^8\xi^7|z|^2e^{2s\alpha}\,dxdt+\int_Q|g|^2e^{2s\alpha}\,dxdt\right)
\end{align*}
for any $\lambda\geq \hat{\lambda}$  and any $s\geq s_0(\lambda)(\sqrt{T}+T).$ 
\end{lemma}
In what follows, we give the definition of $L^2(Q)$-solutions defined by transposition.
\begin{definition}\label{de2.1}
A function $z\in L^2(Q)$ is called a weak solution to problem 
\begin{equation}\label{eq2.4}
\begin{cases}
P^*z=-\frac{\partial z}{\partial t}+\Delta^2 z+\sum_{i,j=1}^n\frac{\partial^2 (D_{ij}z)}{\partial x_i\partial x_j}+\Delta (a_1z)=g,\,\,\,\,\textit{in}\,\,\,Q,\\
z=\Delta z=0,\,\,\,\,\,\textit{on}\,\,\,\,\Sigma,\\
z(x,T)=z_0(x),\,\,\,\,\textit{in}\,\,\,\,\mathcal{O},
\end{cases}
\end{equation}
if for any $w\in L^2(0,T;H^2(\mathcal{O})\cap H_0^1(\mathcal{O}))$ with $Lw\in L^2(Q),$ $w|_{\partial\mathcal{O}}=\Delta w|_{\partial\mathcal{O}}=0$ and $w(x,0)=0$ for any $x\in\mathcal{O},$ the following equality holds true:
\begin{align*}
(z,Lw)_{L^2(Q)}=\int_Q \left(gw-D:\nabla^2w z-a_1z\Delta w\right)\,dxdt+(z_0,w(x,T))_{L^2(\mathcal{O})}.
\end{align*}
\end{definition}
\section{\bf Global Carleman estimates for $L^2$-solutions defined by transposition}
\def\theequation{3.\arabic{equation}}\makeatother
\setcounter{equation}{0}
In this subsection, we mainly establish the global Carleman estiamtes of problem \eqref{eq2.4}. To this purpose, we define
\begin{align}\label{3.1}
J(w,u)=\frac{1}{2}\int_Q|w|^2e^{-2s\alpha}\,dxdt+\frac{1}{2}\int_{Q_\omega}\frac{|u|^2}{s^7\lambda^8\xi^7}e^{-2s\alpha}\,dxdt.
\end{align}
Now, we will consider the following extremal problem:
\begin{align}\label{3.2}
\inf_{(w,u)\in\mathcal{U}}J(w,u),
\end{align}
where $\mathcal{U}$ is the totality of $(w,u)\in L^2(0,T;H^2(\mathcal{O})\cap H_0^1(\mathcal{O}))\times L^2(Q)$ satisfying
\begin{equation}\label{3.3}
\begin{cases}
Lw=\frac{\partial w}{\partial t}+\Delta^2 w=s^6\lambda^8\xi^6e^{2s\alpha}z+\chi_\omega u,\,\,\,\,\forall\,\,\,(x,t)\in Q,\\
w=\Delta w=0,\,\,\,\,\,\forall\,\,\,\,(x,t)\in\Sigma,\\
w(x,0)=w(x,T)=0,\,\,\,\,\forall\,\,\,\,x\in\mathcal{O}.
\end{cases}
\end{equation}
\begin{proposition}\label{3.4}
For any $z\in L^2(Q),$ there exists a unique solution $(\hat{w},\hat{u})\in L^2(0,T;H^2(\mathcal{O})\cap H_0^1(\mathcal{O}))\times L^2(Q)$ to the extremal problem \eqref{3.1}-\eqref{3.3} and the following conclusion holds:
\begin{align*}
&\int_Q\left(\frac{|\nabla\hat{w}|^2}{(s\lambda\xi)^2}+\frac{|\Delta\hat{w}|^2}{(s\lambda\xi)^4}+\frac{|\nabla^2\hat{w}|^2}{(s\lambda\xi)^4}+|\hat{w}|^2\right)e^{-2s\alpha}\,dxdt+\int_{Q_\omega}\frac{|\hat{u}|^2}{s^7\lambda^8\xi^7}e^{-2s\alpha}\,dxdt\\
\leq&C\int_Qs^6\lambda^8\xi^6|z|^2e^{2s\alpha}\,dxdt
\end{align*}
for any $\lambda\geq \hat{\lambda}$  and any $s\geq s_0(\lambda)(\sqrt{T}+T).$
\end{proposition}
\begin{proof}
To begin with, we can prove the existence of a unique solution $(\hat{w},\hat{u})\in L^2(0,T;H^2(\mathcal{O})\cap H_0^1(\mathcal{O}))\times L^2(Q)$ to the extremal problem \eqref{3.1}-\eqref{3.3} by the standard arguments of optimal control of partial differential equations (see \cite{avm, ljl}).

Applying the Lagrange principle (see \cite{avm}) to problem \eqref{3.1}-\eqref{3.3}, we obtain the optimality system for this problem:
\begin{equation}\label{3.5}
\begin{cases}
L\hat{w}=\frac{\partial \hat{w}}{\partial t}+\Delta^2 \hat{w}=s^6\lambda^8\xi^6e^{2s\alpha}z+\chi_\omega \hat{u},\,\,\,\,\textit{in}\,\,\,Q,\\
\hat{w}=\Delta \hat{w}=0,\,\,\,\,\,\textit{on}\,\,\,\,\Sigma,\\
\hat{w}(x,0)=\hat{w}(x,T)=0,\,\,\,\,\textit{in}\,\,\,\,\mathcal{O}
\end{cases}
\end{equation}
and
\begin{equation}\label{3.6}
\begin{cases}
L^*p=-\frac{\partial p}{\partial t}+\Delta^2 p=\hat{w}e^{-2s\alpha},\,\,\,\,\textit{in}\,\,\,Q,\\
p=\Delta p=0,\,\,\,\,\,\textit{on}\,\,\,\,\Sigma,\\
p(x,T)=0,\,\,\,\,\textit{in}\,\,\,\,\mathcal{O},\\
p+\frac{\hat{u}}{s^7\lambda^8\xi^7}e^{-2s\alpha}=0,\,\,\,\,\textit{in}\,\,\,\,Q_\omega.
\end{cases}
\end{equation}
Employing Lemma \ref{le2.2} to problem \eqref{3.6}, yields
\begin{align}\label{3.7}
\nonumber\int_Qs^6\lambda^8\xi^6|p|^2e^{2s\alpha}\,dxdt\leq &C\left(\int_{Q_\omega}s^7\lambda^8\xi^7|p|^2e^{2s\alpha}\,dxdt+\int_Q|\hat{w}|^2e^{-2s\alpha}\,dxdt\right)\\
\nonumber=&C\left(\int_{Q_\omega}\frac{|\hat{u}|^2}{s^7\lambda^8\xi^7}e^{-2s\alpha}\,dxdt+\int_Q|\hat{w}|^2e^{-2s\alpha}\,dxdt\right)\\
=&2CJ(\hat{w},\hat{u})
\end{align}
for any $\lambda\geq \hat{\lambda}$  and any $s\geq s_0(\lambda)(\sqrt{T}+T).$

Combining problem \eqref{3.5} with problem \eqref{3.6}, we obtain
\begin{align*}
0=&\int_Q\left(L\hat{w}-s^6\lambda^8\xi^6e^{2s\alpha}z-\chi_\omega \hat{u}\right)p\,dxdt\\
=&\int_QL^*p\hat{w}\,dxdt-\int_Qs^6\lambda^8\xi^6e^{2s\alpha}zp\,dxdt+\int_{Q_\omega}\frac{|\hat{u}|^2}{s^7\lambda^8\xi^7}e^{-2s\alpha}\,dxdt\\
=&\int_Q|\hat{w}|^2e^{-2s\alpha}\,dxdt-\int_Qs^6\lambda^8\xi^6e^{2s\alpha}zp\,dxdt+\int_{Q_\omega}\frac{|\hat{u}|^2}{s^7\lambda^8\xi^7}e^{-2s\alpha}\,dxdt,
\end{align*}
which implies that
\begin{align*}
2J(\hat{w},\hat{u})=&\int_Q|\hat{w}|^2e^{-2s\alpha}\,dxdt+\int_{Q_\omega}\frac{|\hat{u}|^2}{s^7\lambda^8\xi^7}e^{-2s\alpha}\,dxdt\\
=&\int_Qs^6\lambda^8\xi^6zpe^{2s\alpha}\,dxdt.
\end{align*}
Thus, we deduce from H\"{o}lder's inequality and inequality \eqref{3.7} that
\begin{align*}
2J(\hat{w},\hat{u})=&\int_Q|\hat{w}|^2e^{-2s\alpha}\,dxdt+\int_{Q_\omega}\frac{|\hat{u}|^2}{s^7\lambda^8\xi^7}e^{-2s\alpha}\,dxdt\\
=&\left(\int_Qs^6\lambda^8\xi^6|z|^2e^{2s\alpha}\,dxdt\right)^{\frac{1}{2}}\left(\int_Qs^6\lambda^8\xi^6|p|^2e^{2s\alpha}\,dxdt\right)^{\frac{1}{2}}\\
\leq&C\left(\int_Qs^6\lambda^8\xi^6|z|^2e^{2s\alpha}\,dxdt\right)^{\frac{1}{2}}\sqrt{J(\hat{w},\hat{u})}
\end{align*}
for any $\lambda\geq \hat{\lambda}$  and any $s\geq s_0(\lambda)(\sqrt{T}+T),$ which entails that
\begin{align}\label{3.8}
J(\hat{w},\hat{u})\leq C\int_Qs^6\lambda^8\xi^6|z|^2e^{2s\alpha}\,dxdt
\end{align} 
for any $\lambda\geq \hat{\lambda}$  and any $s\geq s_0(\lambda)(\sqrt{T}+T).$

Multiplying the first equation of problem \eqref{3.5} by $\frac{\hat{w}}{(s\lambda\xi)^4}e^{-2s\alpha}$  and integrating by parts, we get
\begin{align}\label{3.9}
\int_Q\left(\frac{\partial \hat{w}}{\partial t}+\Delta^2 \hat{w}-s^6\lambda^8\xi^6e^{2s\alpha}z-\chi_\omega \hat{u}\right)\frac{\hat{w}}{(s\lambda\xi)^4}e^{-2s\alpha}\,dxdt=0.
\end{align}
In what follows, we will estimate each term of the right hand side of equality \eqref{3.9} by using H\"{o}lder's inequality, Young's inequality, inequality \eqref{3.8} along with the properties of the weight function \eqref{eq2.2}.
\begin{align*}
\int_Q\frac{\partial \hat{w}}{\partial t}\frac{\hat{w}}{(s\lambda\xi)^4}e^{-2s\alpha}\,dxdt=-&\frac{1}{2}\int_Q\frac{\partial}{\partial t}\left(\frac{1}{(s\lambda\xi)^4}e^{-2s\alpha}\right)|\hat{w}|^2\,dxdt\\
=&\frac{1}{2}\int_Q\left(\frac{4\xi_t}{(s\lambda)^4\xi^5}+\frac{2\alpha_t}{(\lambda\xi)^4s^3}\right)|\hat{w}|^2e^{-2s\alpha}\,dxdt,
\end{align*}
which entails that
\begin{align}\label{3.11}
\nonumber\left|\int_Q\frac{\partial \hat{w}}{\partial t}\frac{\hat{w}}{(s\lambda\xi)^4}e^{-2s\alpha}\,dxdt\right|\leq &C\int_Q|\hat{w}|^2e^{-2s\alpha}\,dxdt\\
\leq&C\int_Qs^6\lambda^8\xi^6|z|^2e^{2s\alpha}\,dxdt,
\end{align}

\begin{align}\label{3.12}
\nonumber\left|-\int_Qs^6\lambda^8\xi^6e^{2s\alpha}z\frac{\hat{w}}{(s\lambda\xi)^4}e^{-2s\alpha}\,dxdt\right|=&\left|\int_Qs^2\lambda^4\xi^2z\hat{w}\,dxdt\right|\\
\nonumber\leq&\left(\int_Qs^4\lambda^8\xi^4|z|^2e^{2s\alpha}\,dxdt\right)^{\frac{1}{2}}\left(\int_Q|\hat{w}|^2e^{-2s\alpha}\,dxdt\right)^{\frac{1}{2}}\\
\leq&C\int_Qs^6\lambda^8\xi^6|z|^2e^{2s\alpha}\,dxdt,
\end{align}
\begin{align}\label{3.13}
\nonumber\left|-\int_Q\chi_\omega \hat{u}\frac{\hat{w}}{(s\lambda\xi)^4}e^{-2s\alpha}\,dxdt\right|\leq&\left(\int_{Q_\omega}\frac{|\hat{u}|^2}{s^7\lambda^8\xi^7}e^{-2s\alpha}\,dxdt\right)^{\frac{1}{2}}\left(\int_{Q_\omega}\frac{|\hat{w}|^2}{s\xi}e^{-2s\alpha}\,dxdt\right)^{\frac{1}{2}}\\
\leq&C\int_Qs^6\lambda^8\xi^6|z|^2e^{2s\alpha}\,dxdt,
\end{align}

\begin{align}\label{3.14}
\nonumber&\int_Q\Delta^2 \hat{w}\frac{\hat{w}}{(s\lambda\xi)^4}e^{-2s\alpha}\,dxdt=-\int_Q\nabla\Delta\hat{w}\cdot\left(\frac{\nabla\hat{w}}{(s\lambda\xi)^4}-4\frac{\hat{w}\nabla\eta}{(s\xi)^4\lambda^3}-2\frac{\hat{w}\nabla\eta}{(s\lambda\xi)^3}\right)e^{-2s\alpha}\,dxdt\\
\nonumber=&\int_Q\frac{|\Delta\hat{w}|^2}{(s\lambda\xi)^4}e^{-2s\alpha}\,dxdt+\int_Q\Delta\hat{w}\left(-\frac{8\nabla\hat{w}\cdot\nabla\eta}{(s\xi)^4\lambda^3}-\frac{4\nabla\hat{w}\cdot\nabla\eta}{(s\lambda\xi)^3}-4\frac{\hat{w}\Delta\eta}{(s\xi)^4\lambda^3}\right)e^{-2s\alpha}\,dxdt\\
&+\int_Q\Delta\hat{w}\left(8\frac{\hat{w}|\nabla\eta|^2}{(s\xi)^3\lambda^2}+16\frac{\hat{w}|\nabla\eta|^2}{(s\xi)^4\lambda^2}-2\frac{\hat{w}\Delta\eta}{(s\lambda\xi)^3}+6\frac{\hat{w}|\nabla\eta|^2}{(s\xi)^3\lambda^2}+4\frac{\hat{w}|\nabla\eta|^2}{(s\lambda\xi)^2}\right)e^{-2s\alpha}\,dxdt.
\end{align}
We infer from equality \eqref{3.9}, inequalities \eqref{3.11}-\eqref{3.14} that for any $\lambda\geq \hat{\lambda}$  and any $s\geq s_0(\lambda)(\sqrt{T}+T),$
\begin{align*}
&\int_Q\frac{|\Delta\hat{w}|^2}{(s\lambda\xi)^4}e^{-2s\alpha}\,dxdt\leq\left|\int_Q\Delta\hat{w}\left(-\frac{8\nabla\hat{w}\cdot\nabla\eta}{(s\xi)^4\lambda^3}-\frac{4\nabla\hat{w}\cdot\nabla\eta}{(s\lambda\xi)^3}-4\frac{\hat{w}\Delta\eta}{(s\xi)^4\lambda^3}\right)e^{-2s\alpha}\,dxdt\right|\\
&+\left|\int_Q\Delta\hat{w}\left(8\frac{\hat{w}|\nabla\eta|^2}{(s\xi)^3\lambda^2}+16\frac{\hat{w}|\nabla\eta|^2}{(s\xi)^4\lambda^2}-2\frac{\hat{w}\Delta\eta}{(s\lambda\xi)^3}+6\frac{\hat{w}|\nabla\eta|^2}{(s\xi)^3\lambda^2}+4\frac{\hat{w}|\nabla\eta|^2}{(s\lambda\xi)^2}\right)e^{-2s\alpha}\,dxdt\right|\\
&+C\int_Qs^6\lambda^8\xi^6|z|^2e^{2s\alpha}\,dxdt\\
\leq&\frac{1}{2}\int_Q\frac{|\Delta\hat{w}|^2}{(s\lambda\xi)^4}e^{-2s\alpha}\,dxdt+C\int_Q\left(\frac{|\nabla\hat{w}|^2}{(s\lambda\xi)^2}+|\hat{w}|^2\right)e^{-2s\alpha}\,dxdt+C\int_Qs^6\lambda^8\xi^6|z|^2e^{2s\alpha}\,dxdt\\
\leq&\frac{1}{2}\int_Q\frac{|\Delta\hat{w}|^2}{(s\lambda\xi)^4}e^{-2s\alpha}\,dxdt+C\int_Q\frac{|\nabla\hat{w}|^2}{(s\lambda\xi)^2}e^{-2s\alpha}\,dxdt+C\int_Qs^6\lambda^8\xi^6|z|^2e^{2s\alpha}\,dxdt,
\end{align*}
which implies that there exists a generic positive constant $C_0,$ such that for any $\lambda\geq \hat{\lambda}$  and any $s\geq s_0(\lambda)(\sqrt{T}+T),$
\begin{align}\label{3.15}
\int_Q\frac{|\Delta\hat{w}|^2}{(s\lambda\xi)^4}e^{-2s\alpha}\,dxdt
\leq C_0\int_Q\frac{|\nabla\hat{w}|^2}{(s\lambda\xi)^2}e^{-2s\alpha}\,dxdt+C\int_Qs^6\lambda^8\xi^6|z|^2e^{2s\alpha}\,dxdt.
\end{align}
Thanks to 
\begin{align*}
&\int_Q\frac{|\nabla\hat{w}|^2}{(s\lambda\xi)^2}e^{-2s\alpha}\,dxdt=-\int_Q\frac{\hat{w}\Delta\hat{w}}{(s\lambda\xi)^2}e^{-2s\alpha}\,dxdt+2\int_Q\left(\frac{\hat{w}\nabla\hat{w}\cdot\nabla\eta}{s\lambda\xi}+\frac{\hat{w}\nabla\hat{w}\cdot\nabla\eta}{(s\xi)^2\lambda}\right)e^{-2s\alpha}\,dxdt\\
\leq&\frac{1}{4C_0}\int_Q\frac{|\Delta\hat{w}|^2}{(s\lambda\xi)^4}e^{-2s\alpha}\,dxdt+C\int_Q|\hat{w}|^2e^{-2s\alpha}\,dxdt+\frac{1}{2}\int_Q\frac{|\nabla\hat{w}|^2}{(s\lambda\xi)^2}e^{-2s\alpha}\,dxdt\\
\leq&\frac{1}{4C_0}\int_Q\frac{|\Delta\hat{w}|^2}{(s\lambda\xi)^4}e^{-2s\alpha}\,dxdt+C\int_Qs^6\lambda^8\xi^6|z|^2e^{2s\alpha}\,dxdt+\frac{1}{2}\int_Q\frac{|\nabla\hat{w}|^2}{(s\lambda\xi)^2}e^{-2s\alpha}\,dxdt,
\end{align*}
we obtain
\begin{align}\label{3.16}
\int_Q\frac{|\nabla\hat{w}|^2}{(s\lambda\xi)^2}e^{-2s\alpha}\,dxdt
\leq\frac{1}{2C_0}\int_Q\frac{|\Delta\hat{w}|^2}{(s\lambda\xi)^4}e^{-2s\alpha}\,dxdt+C\int_Qs^6\lambda^8\xi^6|z|^2e^{2s\alpha}\,dxdt.
\end{align}
Therefore, we obtain 
\begin{align}\label{3.17}
\int_Q\frac{|\nabla\hat{w}|^2}{(s\lambda\xi)^2}e^{-2s\alpha}\,dxdt+\int_Q\frac{|\Delta\hat{w}|^2}{(s\lambda\xi)^4}e^{-2s\alpha}\,dxdt\leq C\int_Qs^6\lambda^8\xi^6|z|^2e^{2s\alpha}\,dxdt
\end{align}
for any $\lambda\geq \hat{\lambda}$  and any $s\geq s_0(\lambda)(\sqrt{T}+T).$

Denote by $\tilde{w}=\frac{\hat{w}e^{-s\alpha}}{(s\lambda\xi)^2 },$ then
\begin{align}\label{3.18}
\frac{\partial\tilde{w}}{\partial x_i}=\left(\frac{1}{(s\lambda\xi)^2}\frac{\partial \hat{w}}{\partial x_i}-\frac{1}{s\lambda\xi}\frac{\partial\eta}{\partial x_i}\hat{w}-\frac{2}{s^2\lambda\xi^2}\frac{\partial\eta}{\partial x_i}\hat{w}\right)e^{-s\alpha}
\end{align}
and
\begin{align}\label{3.19}
\nonumber\frac{\partial^2\tilde{w}}{\partial x_i\partial x_j}=&\left(\frac{4}{(s\xi)^2}\frac{\partial\eta}{\partial x_i}\frac{\partial \eta}{\partial x_j}\hat{w}-\frac{2}{s^2\lambda\xi^2}\frac{\partial^2\eta}{\partial x_i\partial x_j}\hat{w}-\frac{1}{s\lambda\xi}\frac{\partial\eta}{\partial x_i}\frac{\partial \hat{w}}{\partial x_j}-\frac{1}{s\lambda\xi}\frac{\partial\eta}{\partial x_j}\frac{\partial \hat{w}}{\partial x_i}\right)e^{-s\alpha}\\
\nonumber&+\left(\frac{1}{(s\lambda\xi)^2}\frac{\partial^2 \hat{w}}{\partial x_i\partial x_j}+\frac{3}{s\xi}\frac{\partial\eta}{\partial x_i}\frac{\partial \eta}{\partial x_j}\hat{w}-\frac{2}{s^2\lambda\xi^2}\frac{\partial\eta}{\partial x_i}\frac{\partial \hat{w}}{\partial x_j}-\frac{2}{s^2\lambda\xi^2}\frac{\partial \hat{w}}{\partial x_i}\frac{\partial \eta}{\partial x_j}\right)e^{-s\alpha}\\
&+\left(\frac{\partial \eta}{\partial x_i}\frac{\partial \eta}{\partial x_j}\hat{w}-\frac{1}{s\lambda\xi}\frac{\partial^2 \eta}{\partial x_i\partial x_j}\hat{w}\right)e^{-s\alpha}.
\end{align}
Therefore, we have
\begin{equation*}
\begin{cases}
\Delta\tilde{w}=\left(\frac{4}{(s\xi)^2}|\nabla\eta|^2\hat{w}-\frac{2}{s^2\lambda\xi^2}\Delta\eta \hat{w}-\frac{2}{s\lambda\xi}\nabla\eta\cdot\nabla \hat{w}+\frac{1}{(s\lambda\xi)^2}\Delta \hat{w}\right)e^{-s\alpha}\\
+\left(\frac{3}{s\xi}|\nabla\eta|^2\hat{w}-\frac{4}{s^2\lambda\xi^2}\nabla \hat{w}\cdot\nabla\eta+|\nabla\eta|^2\hat{w}-\frac{1}{s\lambda\xi}\Delta\eta \hat{w}\right)e^{-s\alpha},\,\,\,\,\textit{in}\,\,\,\,Q,\\
\tilde{w}=0,\,\,\,\,\,\,\,\textit{on}\,\,\,\,\,\Sigma.
\end{cases}
\end{equation*}
From the regularity theory of second order elliptic equations, we conclude that
\begin{align}\label{3.20}
\int_Q\sum_{i,j=1}^n\left|\frac{\partial^2 \tilde{w}}{\partial x_i\partial x_j}\right|^2\,dxdt\leq C\int_Q\left(|\hat{w}|^2+\frac{1}{(s\lambda\xi)^2}|\nabla \hat{w}|^2\right)e^{-2s\alpha}\,dxdt.
\end{align}
Along with inequalities \eqref{3.8}, \eqref{3.17}, \eqref{3.19}-\eqref{3.20}, yields
\begin{align}\label{3.21}
\nonumber&\int_Q\frac{|\nabla^2 \hat{w}|^2}{(s\lambda\xi)^4}e^{-2s\alpha}\,dxdt=\int_Q\sum_{i,j=1}^n\frac{1}{(s\lambda\xi)^4}\left|\frac{\partial^2 \hat{w}}{\partial x_i\partial x_j}\right|^2e^{-2s\alpha}\,dxdt\\
\nonumber\leq &\int_Q\sum_{i,j=1}^n\left|\frac{\partial^2 \tilde{w}}{\partial x_i\partial x_j}\right|^2\,dxdt+C\int_Q\left(|\hat{w}|^2+\frac{1}{(s\lambda\xi)^2}|\nabla \hat{w}|^2\right)e^{-2s\alpha}\,dxdt\\
\nonumber\leq&C\int_Q\left(|\hat{w}|^2+\frac{1}{(s\lambda\xi)^2}|\nabla \hat{w}|^2\right)e^{-2s\alpha}\,dxdt\\
\leq &C\int_Qs^6\lambda^8\xi^6|z|^2e^{2s\alpha}\,dxdt
\end{align}
for any $\lambda\geq \hat{\lambda}$  and any $s\geq s_0(\lambda)(\sqrt{T}+T).$
\end{proof}
In what follows, we will prove the first main result of this paper.
\begin{theorem}\label{3.22}
Assume that $z_0\in L^2(\mathcal{O}),$ $g=g_0+\sum_{i=1}^n\frac{\partial g_i}{\partial x_i}$ with $g_i\in L^2(Q)$ for any $0\leq i\leq n,$ the functions $\alpha,$ $\xi$ be defined by \eqref{eq2.1}. Then there exists a constant $\hat{\lambda}>0$ such that for an arbitrary $\lambda\geq \hat{\lambda},$ we can choose $s_0=s_0(\lambda)>0$ satisfying: there exists a constant $C=C(\lambda)>0$ independent of $s,$ such that the solution $z\in L^2(Q)$ to problem \eqref{eq2.4} satisfies the following inequality:
\begin{align*}
&\int_Qe^{2s\alpha}\left(s^6\lambda^8\xi^6|z|^2+s^4\lambda^6\xi^4|\nabla z|^2+s^2\lambda^4\xi^2|\Delta z|^2+s^2\lambda^4\xi^2|\nabla^2z|^2\right)\,dxdt\\
\leq&C\int_Q\left(|g_0|^2+\sum_{i=1}^n(s\lambda\xi)^2|g_i|^2\right)e^{2s\alpha}\,dxdt+C\int_{Q_\omega}s^7\lambda^8\xi^7|z|^2e^{2s\alpha}\,dxdt
\end{align*}
for any $\lambda\geq \hat{\lambda}(1+\|D\|_{L^{\infty}(Q)}^{\frac{1}{2}}+\|a_1\|_{L^{\infty}(Q)}^{\frac{1}{2}})$  and any $s\geq s_0(\lambda)(\sqrt{T}+T).$ 
\end{theorem}
\begin{proof}
Let $(\hat{w},\hat{u})$ be the solution of the extremal problem \eqref{3.1}-\eqref{3.3}
obtained in Proposition \ref{3.4}. Then by definition of weak solution, we have
\begin{align*}
0=&\int_Q \left(zL\hat{w}-g\hat{w}+D:\nabla^2\hat{w} z+a_1z\Delta \hat{w}\right)\,dxdt\\
=&\int_Q\left(s^6\lambda^8\xi^6e^{2s\alpha}|z|^2+\chi_\omega uz-g\hat{w}+D:\nabla^2\hat{w} z+a_1z\Delta \hat{w}\right).
\end{align*}
Therefore, we conclude from H\"{o}lder's inequality and Proposition \ref{3.4} that
\begin{align*}
&\int_Qs^6\lambda^8\xi^6e^{2s\alpha}|z|^2\,dxdt=\int_Q\left(g\hat{w}-D:\nabla^2w z-a_1z\Delta w\right)\,dxdt-\int_{Q_\omega} \hat{u}z\,dxdt\\
=&\int_Q\left(g_0\hat{w}-D:\nabla^2w z-a_1z\Delta w-\sum_{i=1}^ng_i\frac{\partial\hat{w}}{\partial x_i}\right)\,dxdt-\int_{Q_\omega} \hat{u}z\,dxdt\\
\leq&\|g_0e^{s\alpha}\|_{L^2(Q)}\|\hat{w}e^{-s\alpha}\|_{L^2(Q)}+\sum_{i=1}^n\|(s\lambda\xi)g_ie^{s\alpha}\|_{L^2(Q)}\|\frac{\hat{w}_{x_i}}{s\lambda\xi}e^{-s\alpha}\|_{L^2(Q)}\\
&+\|D\|_{L^{\infty}(Q)}\left\|\frac{|\nabla^2\hat{w}|}{(s\lambda\xi)^2}e^{-s\alpha}\right\|_{L^2(Q)}\|(s\lambda\xi)^2ze^{s\alpha}\|_{L^2(Q)}+\|a_1\|_{L^{\infty}(Q)}\left\|\frac{|\Delta\hat{w}|}{(s\lambda\xi)^2}e^{-s\alpha}\right\|_{L^2(Q)}\|(s\lambda\xi)^2ze^{s\alpha}\|_{L^2(Q)}\\
&+\left(\int_{Q_\omega}\frac{|\hat{u}|^2}{s^7\lambda^8\xi^7}e^{-2s\alpha}\,dxdt\right)^{\frac{1}{2}}
\left(\int_{Q_\omega}s^7\lambda^8\xi^7|z|^2e^{2s\alpha}\,dxdt\right)^{\frac{1}{2}}\\
\leq&C\|g_0e^{s\alpha}\|_{L^2(Q)}^2+C\sum_{i=1}^n\|(s\lambda\xi)g_ie^{s\alpha}\|_{L^2(Q)}^2+
C\int_{Q_\omega}s^7\lambda^8\xi^7|z|^2e^{2s\alpha}\,dxdt\\
&+\frac{1}{2}\int_Qs^6\lambda^8\xi^6e^{2s\alpha}|z|^2\,dxdt,
\end{align*}
for any $\lambda\geq \hat{\lambda}(\lambda)(1+\|D\|_{L^{\infty}(Q)}^{\frac{1}{2}}+\|a_1\|_{L^{\infty}(Q)}^{\frac{1}{2}})$  and any $s\geq s_0(\sqrt{T}+T),$ which implies that
\begin{align}\label{3.23}
\nonumber\int_Qs^6\lambda^8\xi^6e^{2s\alpha}|z|^2\,dxdt\leq&C\int_Q\left(|g_0|^2+\sum_{i=1}^n(s\lambda\xi)^2|g_i|^2\right)e^{2s\alpha}\,dxdt\\
&+C\int_{Q_\omega}s^7\lambda^8\xi^7|z|^2e^{2s\alpha}\,dxdt
\end{align}
for any $\lambda\geq \hat{\lambda}(1+\|D\|_{L^{\infty}(Q)}^{\frac{1}{2}}+\|a_1\|_{L^{\infty}(Q)}^{\frac{1}{2}})$  and any $s\geq s_0(\lambda)(\sqrt{T}+T).$

Taking the inner product of the first equation of problem \eqref{eq2.4} with $s^2\lambda^4\xi^2ze^{2s\alpha}$ in $L^2(Q),$ we obtain
\begin{align}\label{3.24}
\int_Q\left(-\frac{\partial z}{\partial t}+\Delta^2z+\sum_{i,j=1}^n\frac{\partial^2 (D_{ij}z)}{\partial x_i\partial x_j}+\Delta (a_1z)-g\right)s^2\lambda^4\xi^2ze^{2s\alpha}\,dxdt=0.
\end{align}
In what follows, we will estimate each term of the right hand side of equality \eqref{3.24} by using H\"{o}lder's inequality, Young's inequality, inequality \eqref{3.23} along with the properties of the weight function \eqref{eq2.2}.
\begin{align}\label{3.25}
\nonumber\left|-\int_Q\frac{\partial z}{\partial t}s^2\lambda^4\xi^2ze^{2s\alpha}\,dxdt\right|=&\frac{1}{2}\left|\int_Q2s^2\lambda^4\xi\xi_t|z|^2e^{2s\alpha}+2s^3\lambda^4\xi^2\alpha_t|z|^2e^{2s\alpha}\,dxdt\right|\\
\leq&C\int_Qs^6\lambda^8\xi^6e^{2s\alpha}|z|^2\,dxdt.
\end{align}
\begin{align*}
&\left|-\int_Qgs^2\lambda^4\xi^2ze^{2s\alpha}\,dxdt\right|\leq\|g_0e^{s\alpha}\|_{L^2(Q)}\left(\int_Qs^4\lambda^8\xi^4e^{2s\alpha}|z|^2\,dxdt\right)^{\frac{1}{2}}+\sum_{i=1}^n\left|\int_Qg_is^2\lambda^4\xi^2z_{x_i}e^{2s\alpha}\,dxdt\right|\\
&+2\sum_{i=1}^n\left|\int_Qg_is^2\lambda^4\xi\xi_{x_i}ze^{2s\alpha}\,dxdt\right|+2\sum_{i=1}^n\left|\int_Qg_is^3\lambda^4\xi^2\alpha_{x_i}ze^{2s\alpha}\,dxdt\right|\\
\leq&\|g_0e^{s\alpha}\|_{L^2(Q)}\left(\int_Qs^6\lambda^8\xi^6e^{2s\alpha}|z|^2\,dxdt\right)^{\frac{1}{2}}+\left(\sum_{i=1}^n\int_Q(s\lambda\xi)^2|g_i|^2e^{2s\alpha}\,dxdt\right)^{\frac{1}{2}}\left(\int_Qs^2\lambda^6\xi^2|\nabla z|^2e^{2s\alpha}\,dxdt\right)^{\frac{1}{2}}\\
&+C\left(\sum_{i=1}^n\int_Q(s\lambda\xi)^2|g_i|^2e^{2s\alpha}\,dxdt\right)^{\frac{1}{2}}\left(\int_Qs^6\lambda^8\xi^6|z|^2e^{2s\alpha}\,dxdt\right)^{\frac{1}{2}},
\end{align*}
which implies that
\begin{align}\label{3.26}
\nonumber\left|-\int_Qgs^2\lambda^4\xi^2ze^{2s\alpha}\,dxdt\right|\leq &C\int_Q\left(|g_0|^2+\sum_{i=1}^n(s\lambda\xi)^2|g_i|^2\right)e^{2s\alpha}\,dxdt+C\int_Qs^6\lambda^8\xi^6e^{2s\alpha}|z|^2\,dxdt\\
&+C\int_Qs^2\lambda^6\xi^2|\nabla z|^2e^{2s\alpha}\,dxdt.
\end{align}

\begin{align*}
&\int_Q\Delta^2zs^2\lambda^4\xi^2ze^{2s\alpha}\,dxdt=-\int_Q\nabla\Delta z\cdot\left(s^2\lambda^4\xi^2\nabla z +2s^2\lambda^5\xi^2\nabla\eta z +2s^3\lambda^5\xi^3\nabla\eta z \right)e^{2s\alpha}\,dxdt\\
=&\int_Qs^2\lambda^4\xi^2|\Delta z|^2 e^{2s\alpha}\,dxdt+\int_Q\Delta z\left(4s^2\lambda^5\xi^2\nabla\eta\cdot\nabla z+4s^3\lambda^5\xi^3\nabla\eta\cdot\nabla z\right)e^{2s\alpha}\,dxdt\\
&+\int_Q\Delta z\left(4s^2\lambda^6\xi^2|\nabla\eta|^2 z+2s^2\lambda^5\xi^2\Delta\eta z+10s^3\lambda^6\xi^3|\nabla\eta|^2 z\right)e^{2s\alpha}\,dxdt\\
&+\int_Q\Delta z\left(2s^3\lambda^5\xi^3\Delta\eta z+4s^4\lambda^6\xi^4|\nabla\eta|^2 z\right)e^{2s\alpha}\,dxdt,
\end{align*}
which entails that
\begin{align}\label{3.27}
\nonumber\int_Qs^2\lambda^4\xi^2|\Delta z|^2 e^{2s\alpha}\,dxdt\leq&\int_Q\Delta^2zs^2\lambda^4\xi^2ze^{2s\alpha}\,dxdt+C\int_Qs^4\lambda^6\xi^4|\nabla z|^2e^{2s\alpha}\,dxdt\\
&+C\int_Qs^6\lambda^8\xi^6|z|^2e^{2s\alpha}\,dxdt.
\end{align}
By integrating by parts, we obtain
\begin{align}\label{3.28}
\nonumber&\int_Q\left(\sum_{i,j=1}^n\frac{\partial^2 (D_{ij}z)}{\partial x_i\partial x_j}+\Delta (a_1z)\right)s^2\lambda^4\xi^2ze^{2s\alpha}\,dxdt\\
=&\int_Q\left(\sum_{i,j=1}^n\frac{\partial^2 (s^2\lambda^4\xi^2ze^{2s\alpha})}{\partial x_i\partial x_j}D_{ij}z+a_1z\Delta (s^2\lambda^4\xi^2ze^{2s\alpha})\right)\,dxdt.
\end{align}
Thanks to
\begin{align*}
\frac{\partial}{\partial x_i}\left(s^2\lambda^4\xi^2ze^{2s\alpha}\right)=\left(s^2\lambda^4\xi^2\frac{\partial z}{\partial x_i}+2s^2\lambda^5\xi^2\frac{\partial \eta}{\partial x_i}z+2s^3\lambda^5\xi^3\frac{\partial \eta}{\partial x_i}z\right)e^{2s\alpha}
\end{align*}
and
\begin{align}\label{3.29}
\nonumber&\frac{\partial^2}{\partial x_i\partial x_j}\left(s^2\lambda^4\xi^2ze^{2s\alpha}\right)=\left(s^2\lambda^4\xi^2\frac{\partial^2 z}{\partial x_i\partial x_j}+2s^2\lambda^5\xi^2\frac{\partial \eta}{\partial x_i}\frac{\partial z}{\partial x_j}+2s^2\lambda^5\xi^2\frac{\partial \eta}{\partial x_j}\frac{\partial z}{\partial x_i}\right)e^{2s\alpha}\\
\nonumber&+\left(2s^3\lambda^5\xi^3\frac{\partial \eta}{\partial x_i}\frac{\partial z}{\partial x_j}+2s^3\lambda^5\xi^3\frac{\partial \eta}{\partial x_j}\frac{\partial z}{\partial x_i}+2s^3\lambda^5\xi^3\frac{\partial^2 \eta}{\partial x_i\partial x_j}z+2s^2\lambda^5\xi^2\frac{\partial^2 \eta}{\partial x_i\partial x_j}z\right)e^{2s\alpha}\\
&+\left(4s^2\lambda^6\xi^2\frac{\partial \eta}{\partial x_i}\frac{\partial \eta}{\partial x_j}z+4s^4\lambda^6\xi^4\frac{\partial \eta}{\partial x_j}\frac{\partial \eta}{\partial x_i}z+10s^3\lambda^6\xi^3\frac{\partial \eta}{\partial x_i}\frac{\partial\eta}{\partial x_j}z\right)e^{2s\alpha}, 
\end{align}
we conclude from inequalities \eqref{3.28}-\eqref{3.29} that
\begin{align}\label{3.30}
\nonumber&\left|\int_Q\left(\sum_{i,j=1}^n\frac{\partial^2 (D_{ij}z)}{\partial x_i\partial x_j}+\Delta (a_1z)\right)s^2\lambda^4\xi^2ze^{2s\alpha}\,dxdt\right|\\
\nonumber\leq&C\int_Q\left(s^2\lambda^4\xi^2|\nabla^2 z|^2+s^4\lambda^6\xi^4|\nabla z|^2+s^6\lambda^8\xi^6|z|^2\right)e^{2s\alpha}\,dxdt\\
&+\frac{1}{2}\int_Qs^2\lambda^4\xi^2|\Delta z|^2 e^{2s\alpha}\,dxdt
\end{align}
for any $\lambda\geq \hat{\lambda}(1+\|D\|_{L^{\infty}(Q)}^{\frac{1}{2}}+\|a_1\|_{L^{\infty}(Q)}^{\frac{1}{2}})$  and any $s\geq s_0(\lambda)(\sqrt{T}+T).$

It follows from inequalities \eqref{3.23}-\eqref{3.27}, \eqref{3.30} that 
\begin{align}\label{3.31}
\nonumber\int_Qs^2\lambda^4\xi^2|\Delta z|^2 e^{2s\alpha}\,dxdt\leq&C\int_Q\left(|g_0|^2+\sum_{i=1}^n(s\lambda\xi)^2|g_i|^2\right)e^{2s\alpha}\,dxdt+C\int_{Q_\omega}s^7\lambda^8\xi^7e^{2s\alpha}|z|^2\,dxdt\\
&+C\int_Q\left(s^2\lambda^4\xi^2|\nabla^2 z|^2+s^4\lambda^6\xi^4|\nabla z|^2\right)e^{2s\alpha}\,dxdt.
\end{align}
Denote by $\tilde{z}=s\lambda^2\xi ze^{s\alpha},$ then we have
\begin{align*}
\frac{\partial\tilde{z}}{\partial x_i}=\left(s\lambda^3\xi\frac{\partial\eta}{\partial x_i}z+s\lambda^2\xi\frac{\partial z}{\partial x_i}+s^2\lambda^3\xi^2\frac{\partial\eta}{\partial x_i}z\right)e^{s\alpha}
\end{align*}
and
\begin{align}\label{3.32}
\nonumber\frac{\partial^2\tilde{z}}{\partial x_i\partial x_j}=&\left(s\lambda^4\xi\frac{\partial\eta}{\partial x_i}\frac{\partial \eta}{\partial x_j}z+s\lambda^3\xi\frac{\partial^2\eta}{\partial x_i\partial x_j}z+s\lambda^3\xi\frac{\partial\eta}{\partial x_i}\frac{\partial z}{\partial x_j}+s\lambda^3\xi\frac{\partial\eta}{\partial x_j}\frac{\partial z}{\partial x_i}\right)e^{s\alpha}\\
\nonumber&+\left(s\lambda^2\xi\frac{\partial^2 z}{\partial x_i\partial x_j}+3s^2\lambda^4\xi^2\frac{\partial\eta}{\partial x_i}\frac{\partial \eta}{\partial x_j}z+s^2\lambda^3\xi^2\frac{\partial\eta}{\partial x_i}\frac{\partial z}{\partial x_j}+s^2\lambda^3\xi^2\frac{\partial z}{\partial x_i}\frac{\partial \eta}{\partial x_j}\right)e^{s\alpha}\\
&+\left(s^3\lambda^4\xi^3\frac{\partial \eta}{\partial x_i}\frac{\partial \eta}{\partial x_j}z+s^2\lambda^3\xi^2\frac{\partial^2 \eta}{\partial x_i\partial x_j}z\right)e^{s\alpha}.
\end{align}
Therefore, we have
\begin{equation*}
\begin{cases}
\Delta\tilde{z}=\left(s\lambda^4\xi|\nabla\eta|^2z+s\lambda^3\xi\Delta\eta z+2s\lambda^3\xi\nabla\eta\cdot\nabla z+s\lambda^2\xi\Delta z\right)e^{s\alpha}\\
+\left(3s^2\lambda^4\xi^2|\nabla\eta|^2z+2s^2\lambda^3\xi^2\nabla z\cdot\nabla\eta+s^3\lambda^4\xi^3|\nabla\eta|^2z+s^2\lambda^3\xi^2\Delta\eta z\right)e^{s\alpha},\,\,\,\,\textit{in}\,\,\,\,Q,\\
\tilde{z}=0,\,\,\,\,\,\,\,\textit{on}\,\,\,\,\,\Sigma.
\end{cases}
\end{equation*}
We conclude from the regularity theory of second order elliptic equations that
\begin{align}\label{3.33}
\int_Q\sum_{i,j=1}^n\left|\frac{\partial^2 \tilde{z}}{\partial x_i\partial x_j}\right|^2\,dxdt\leq C\int_Q\left(s^6\lambda^8\xi^6|z|^2+s^4\lambda^6\xi^4|\nabla z|^2\right)e^{2s\alpha}\,dxdt.
\end{align}
It follows from inequalities \eqref{3.32}-\eqref{3.33} that
\begin{align}\label{3.34}
\nonumber&\int_Qs^2\lambda^4\xi^2|\nabla^2 z|^2e^{2s\alpha}\,dxdt=\int_Q\sum_{i,j=1}^ns^2\lambda^4\xi^2\left|\frac{\partial^2 z}{\partial x_i\partial x_j}\right|^2e^{2s\alpha}\,dxdt\\
\nonumber\leq &\int_Q\sum_{i,j=1}^n\left|\frac{\partial^2 \tilde{z}}{\partial x_i\partial x_j}\right|^2\,dxdt+C\int_Q\left(s^6\lambda^8\xi^6|z|^2+s^4\lambda^6\xi^4|\nabla z|^2\right)e^{2s\alpha}\,dxdt\\
\leq&C\int_Q\left(s^6\lambda^8\xi^6|z|^2+s^4\lambda^6\xi^4|\nabla z|^2\right)e^{2s\alpha}\,dxdt.
\end{align}
Thus, we infer from inequalities \eqref{3.23}, \eqref{3.31}, \eqref{3.34} that there exists a generic positive constant $\beta_1,$ such that
\begin{align}\label{3.35}
\nonumber&\int_Qs^2\lambda^4\xi^2|\Delta z|^2 e^{2s\alpha}\,dxdt+\int_Qs^2\lambda^4\xi^2|\nabla^2 z|^2e^{2s\alpha}\,dxdt+\int_Qs^4\lambda^6\xi^4|\nabla z|^2e^{2s\alpha}\,dxdt\\
\nonumber\leq&C\int_Q\left(|g_0|^2+\sum_{i=1}^n(s\lambda\xi)^2|g_i|^2\right)e^{2s\alpha}\,dxdt+C\int_{Q_\omega}s^7\lambda^8\xi^7e^{2s\alpha}|z|^2\,dxdt\\
&+\beta_1\int_Qs^4\lambda^6\xi^4|\nabla z|^2e^{2s\alpha}\,dxdt
\end{align}
for any $\lambda\geq \hat{\lambda}(1+\|D\|_{L^{\infty}(Q)}^{\frac{1}{2}}+\|a_1\|_{L^{\infty}(Q)}^{\frac{1}{2}})$  and any $s\geq s_0(\lambda)(\sqrt{T}+T).$ 

Thanks to
\begin{align*}
\int_Qs^4\lambda^6\xi^4|\nabla z|^2e^{2s\alpha}\,dxdt=&-\int_Qs^4\lambda^6\xi^4z\Delta ze^{2s\alpha}\,dxdt-4\int_Qs^4\lambda^7\xi^4z\nabla z\cdot\nabla\eta e^{2s\alpha}\,dxdt\\
&-4\int_Qs^5\lambda^7\xi^5z\nabla z\cdot\nabla\eta e^{2s\alpha}\,dxdt\\
\leq&\frac{1}{4\beta_1}\int_Qs^2\lambda^4\xi^2|\Delta z|^2 e^{2s\alpha}\,dxdt+C\int_Qs^6\lambda^8\xi^6|z|^2 e^{2s\alpha}\,dxdt\\
&+\frac{1}{2}\int_Qs^4\lambda^6\xi^4|\nabla z|^2e^{2s\alpha}\,dxdt,
\end{align*}
we obtain
\begin{align}\label{3.36}
\int_Qs^4\lambda^6\xi^4|\nabla z|^2e^{2s\alpha}\,dxdt
\leq\frac{1}{2\beta_1}\int_Qs^2\lambda^4\xi^2|\Delta z|^2 e^{2s\alpha}\,dxdt+C\int_Qs^6\lambda^8\xi^6|z|^2 e^{2s\alpha}\,dxdt.
\end{align}

Therefore, we infer from inequalities \eqref{3.23}, \eqref{3.35}-\eqref{3.36} that
\begin{align}\label{3.37}
\nonumber&\int_Qs^2\lambda^4\xi^2|\Delta z|^2 e^{2s\alpha}\,dxdt+\int_Qs^2\lambda^4\xi^2|\nabla^2 z|^2e^{2s\alpha}\,dxdt+\int_Qs^4\lambda^6\xi^4|\nabla z|^2e^{2s\alpha}\,dxdt\\
\leq&C\int_Q\left(|g_0|^2+\sum_{i=1}^n(s\lambda\xi)^2|g_i|^2\right)e^{2s\alpha}\,dxdt+C\int_{Q_\omega}s^7\lambda^8\xi^7e^{2s\alpha}|z|^2\,dxdt.
\end{align}
for any $\lambda\geq \hat{\lambda}(1+\|D\|_{L^{\infty}(Q)}^{\frac{1}{2}}+\|a_1\|_{L^{\infty}(Q)}^{\frac{1}{2}})$  and any $s\geq s_0(\lambda)(\sqrt{T}+T).$ 
\end{proof}

\section{\bf Null controllability of fourth order semilinear parabolic equations}
\def\theequation{4.\arabic{equation}}\makeatother
\setcounter{equation}{0}
In this section, we mainly consider the null controllability for fourth order semilinear parabolic equations. To begin with, based on the optimal control theory and global Carleman estimates obtained in Theorem \ref{3.22}, we consider the null controllability of the fourth order linear parabolic equations:
\begin{equation}\label{eq4.1}
\begin{cases}
Gy=\frac{\partial y}{\partial t}+\Delta^2 y+a_0y+B_0\cdot\nabla y+D:\nabla^2y+a_1\Delta y=\chi_\omega v+g,\,\,\,\,\textit{in}\,\,\,Q,\\
y=\Delta y=0,\,\,\,\,\,\textit{on}\,\,\,\,\Sigma,\\
y(x,0)=y_0(x),\,\,\,\,\textit{in}\,\,\,\,\mathcal{O}.
\end{cases}
\end{equation}
The null controllability of problem \eqref{eq4.1} is that for any time $T>0,$ we would like to find a control $v\in L^2(Q_\omega),$ such that
\begin{align}\label{eq4.2}
y(x,T)=0,\,\,\,\,\,\,x\in\mathcal{O}.
\end{align}
Now, to study the null controllability of problem \eqref{eq4.1}, we introduce some new weight functions
\begin{equation}\label{eq4.3}
\tilde{\alpha}(x,t)=
\begin{cases}
\alpha(x,\frac{T}{2}),\,\,\,\,\,\forall\,\,\,\,t\in[0,\frac{T}{2}],\\
\alpha(x,t),\,\,\,\,\,\forall\,\,\,\,t\in[\frac{T}{2},T]
\end{cases}
\end{equation}
and
\begin{equation}\label{eq4.4}
\tilde{\xi}(x,t)=
\begin{cases}
\xi(x,\frac{T}{2}),\,\,\,\,\,\forall\,\,\,\,t\in[0,\frac{T}{2}],\\
\xi(x,t),\,\,\,\,\,\forall\,\,\,\,t\in[\frac{T}{2},T].
\end{cases}
\end{equation}
Obviously, $\alpha(x,t)\leq \tilde{\alpha}(x,t)$ and $\xi(x,t)\geq \tilde{\xi}(x,t)$ for any $(x,t)\in\mathcal{O}\times [0,\frac{T}{2}].$

%
%
%

In what follows, we will prove the null controllability for problem \eqref{eq4.1}.
\begin{theorem}\label{th4.1}
Let the functions $\tilde{\xi},$ $\tilde{\alpha}$ be defined by \eqref{eq4.3} and \eqref{eq4.4}, respectively. If $\tilde{\xi}^{-3}e^{-s\tilde{\alpha}}g\in L^2(Q)$ and $y_0\in L^2(\mathcal{O}).$ Then there exists a control $v\in L^2(Q_\omega)$ such that the solution of problem \eqref{eq4.1} satisfies \eqref{eq4.2}. Moreover we can obtain a control $v$ of minimal norm in $L^2(Q_\omega)$ among the admissible controls (such that \eqref{eq4.2} is satisfied).
\end{theorem}
\begin{proof}
Let $\epsilon$ be a strictly positive real number and define
\begin{align}\label{eq4.5}
\mathcal{J}_\epsilon(y,v)=\frac{1}{2\epsilon}\int_{\mathcal{O}}|y(x,T)|^2\,dx+\frac{1}{2}\int_{Q_\omega}|v(x,t)|^2\,dxdt.
\end{align}
Now, we introduce the following extremal problem
\begin{align}\label{eq4.6}
\min_{(y,v)\in\mathcal{V}}\mathcal{J}_\epsilon(y,v),
\end{align}
where $\mathcal{V}$ is the totality of $(y,v)\in Y(Q)\times L^2(Q_\omega)$ satisfying 
\begin{equation}\label{eq4.7}
\begin{cases}
Gy=\frac{\partial y}{\partial t}+\Delta^2 y+a_0y+B_0\cdot\nabla y+D:\nabla^2y+a_1\Delta y=\chi_\omega v+g,\,\,\,\,\textit{in}\,\,\,Q,\\
y=\Delta y=0,\,\,\,\,\,\textit{on}\,\,\,\,\Sigma,\\
y(x,0)=y_0(x),\,\,\,\,\,\,\textit{in}\,\,\,\,\mathcal{O},
\end{cases}
\end{equation}
where
\begin{align*}
Y(Q)=\{y: G(y)\in L^2(Q),y|_{\Sigma}=\Delta y|_{\Sigma}=0, y(0)\in L^2(\mathcal{O})\}
\end{align*}
equipped with the norm
\begin{align*}
\|y\|_{Y(Q)}^2=\|G(y)\|_{L^2(Q)}^2+\|y(0)\|_{L^2(\mathcal{O})}^2.
\end{align*}
By the standard arguments of optimal control of partial differential equations (see \cite{avm, ljl}), one can easily prove the existence of a unique solution $(y_\epsilon,v_\epsilon)\in Y(Q)\times L^2(Q_\omega)$ to the extremal problem \eqref{eq4.5}-\eqref{eq4.7}.

Applying the Lagrange principle (see \cite{avm}) to problem \eqref{eq4.5}-\eqref{eq4.7}, we obtain the following optimality system:
\begin{equation}\label{eq4.8}
\begin{cases}
Gy_\epsilon=\frac{\partial y_\epsilon}{\partial t}+\Delta^2 y_\epsilon+a_0y_\epsilon+B_0\cdot\nabla y_\epsilon+D:\nabla^2 y_\epsilon+a_1\Delta y_\epsilon=g+\chi_\omega v_\epsilon,\,\,\,\,\textit{in}\,\,\,Q,\\
y_\epsilon=\Delta y_\epsilon=0,\,\,\,\,\,\textit{on}\,\,\,\,\Sigma,\\
y_\epsilon(x,0)=y_0(x),\,\,\,\,\textit{in}\,\,\,\,\mathcal{O}
\end{cases}
\end{equation}
and
\begin{equation}\label{eq4.9}
\begin{cases}
G^*p_\epsilon=-\frac{\partial p_\epsilon}{\partial t}+\Delta^2 p_\epsilon+a_0p_\epsilon-\nabla\cdot(B_0p_\epsilon)+\sum_{i,j=1}^n\frac{\partial^2(D_{ij}p_\epsilon)}{\partial x_i\partial x_j}+\Delta (a_1p_\epsilon)=0,\,\,\,\,\textit{in}\,\,\,Q,\\
p_\epsilon=\Delta p_\epsilon=0,\,\,\,\,\,\textit{on}\,\,\,\,\Sigma,\\
p_\epsilon(x,T)=\frac{1}{\epsilon}y_\epsilon(x,T),\,\,\,\,\textit{in}\,\,\,\,\mathcal{O},\\
p_\epsilon+v_\epsilon=0,\,\,\,\,\textit{in}\,\,\,\,Q_\omega.
\end{cases}
\end{equation}
We infer from problems \eqref{eq4.8}-\eqref{eq4.9} that
\begin{align*}
\frac{1}{\epsilon}\int_{\mathcal{O}}|y_\epsilon(T)|^2\,dx-\int_{\mathcal{O}}p_\epsilon(x,0)y_0(x)\,dx=\int_Qg(x,t)p_\epsilon\,dxdt-\int_{Q_\epsilon}|v_\epsilon|^2\,dxdt,
\end{align*}
which entails that
\begin{align}\label{eq4.10}
\nonumber2\mathcal{J}_\epsilon(y_\epsilon,v_\epsilon)=&\frac{1}{\epsilon}\int_{\mathcal{O}}|y_\epsilon(T)|^2\,dx+\int_{Q_\omega}|v_\epsilon|^2\,dxdt\\
=&\int_Qg(x,t)p_\epsilon\,dxdt+\int_{\mathcal{O}}p_\epsilon(x,0)y_0(x)\,dx.
\end{align}
Employing Theorem \ref{3.22} to problem \eqref{eq4.9}, yields
\begin{align*}
&\int_Qe^{2s\alpha}\left(s^6\lambda^8\xi^6|p_\epsilon|^2+s^4\lambda^6\xi^4|\nabla p_\epsilon|^2+s^2\lambda^4\xi^2|\Delta p_\epsilon|^2+s^2\lambda^4\xi^2|\nabla^2p_\epsilon|^2\right)\,dxdt\\
\leq&C\int_Q\left(|a_0|^2|p_\epsilon|^2+(s\lambda\xi)^2|B_0|^2|p_\epsilon|^2\right)e^{2s\alpha}\,dxdt+C\int_{Q_\omega}s^7\lambda^8\xi^7|v_\epsilon|^2e^{2s\alpha}\,dxdt
\end{align*}
for any $\lambda\geq \hat{\lambda}(1+\|D\|_{L^{\infty}(Q)}^{\frac{1}{2}}+\|a_1\|_{L^{\infty}(Q)}^{\frac{1}{2}})$  and any $s\geq s_0(\lambda)(\sqrt{T}+T),$ where $|B_0|=\left(\sum_{j=1}^n|B_{0j}|^2\right)^{\frac{1}{2}}.$

Therefore, we obtain
\begin{align}\label{eq4.11}
\nonumber&\int_Qe^{2s\alpha}\left(s^6\lambda^8\xi^6|p_\epsilon|^2+s^4\lambda^6\xi^4|\nabla p_\epsilon|^2+s^2\lambda^4\xi^2|\Delta p_\epsilon|^2+s^2\lambda^4\xi^2|\nabla^2p_\epsilon|^2\right)\,dxdt\\
\leq&C\int_{Q_\omega}s^7\lambda^8\xi^7|v_\epsilon|^2e^{2s\alpha}\,dxdt
\end{align}
for any $\lambda\geq \hat{\lambda}(1+\|a_0\|_{L^{\infty}(Q)}^{\frac{1}{4}}+\|a_1\|_{L^{\infty}(Q)}^{\frac{1}{2}}+\|B_0\|_{L^{\infty}(Q)}^{\frac{1}{3}}+\|D\|_{L^{\infty}(Q)}^{\frac{1}{2}})$ and any $s\geq s_0(\lambda)(\sqrt{T}+T).$

For any $t\in [0,\frac{T}{2}]$ and any $T\geq s\geq t,$ taking the inner product of the first equation of problem \eqref{eq4.9} with $p_\epsilon$ in $L^2(\mathcal{O}\times (t,s)),$ we obtain for any $t\in [0,\frac{T}{2}]$ and any $T\geq s\geq t,$
\begin{align*}
&\frac{1}{2}\|p_\epsilon(t)\|_{L^2(\mathcal{O})}^2+\int_t^s\|\Delta p_\epsilon(r)\|_{L^2(\mathcal{O})}^2\,dr\\
\leq&\|a_0\|_{L^{\infty}(Q)}\int_t^s\|p_\epsilon(r)\|_{L^2(\mathcal{O})}^2\,dr+\|B_0\|_{L^{\infty}(Q)}\int_t^s\|p_\epsilon(r)\|_{L^2(\mathcal{O})}\|\nabla p_\epsilon(r)\|_{L^2(\mathcal{O})}\,dr\\
&+\|D\|_{L^{\infty}(Q)}\int_t^s\|p_\epsilon(r)\|_{L^2(\mathcal{O})}\|\nabla^2 p_\epsilon(r)\|_{L^2(\mathcal{O})}\,dr+\|a_1\|_{L^{\infty}(Q)}\int_t^s\|p_\epsilon(r)\|_{L^2(\mathcal{O})}\|\Delta p_\epsilon(r)\|_{L^2(\mathcal{O})}\,dr\\
&+\frac{1}{2}\|p_\epsilon(s)\|_{L^2(\mathcal{O})}^2.
\end{align*}
Thanks to
\begin{align*}
\|\nabla p_\epsilon\|_{L^2(\mathcal{O})}^2\leq &\|p_\epsilon\|_{L^2(\mathcal{O})}\|\Delta p_\epsilon\|_{L^2(\mathcal{O})},\\
\|\nabla^2p_\epsilon\|_{L^2(\mathcal{O})}^2\leq &C\left(\|p_\epsilon\|_{L^2(\mathcal{O})}^2+\|\Delta p_\epsilon\|_{L^2(\mathcal{O})}^2\right)
\end{align*}
along with Young's inequality, we conclude that there exists a generic positive constant $\mathcal{L}_1,$ such that for any $t\in [0,\frac{T}{2}]$ and any $T\geq s\geq t,$
\begin{align*}
&\|p_\epsilon(t)\|_{L^2(\mathcal{O})}^2+\int_t^s\|\Delta p_\epsilon(r)\|_{L^2(\mathcal{O})}^2\,dr\\
\leq&\mathcal{L}_1(1+\|a_0\|_{L^{\infty}(Q)}+\|B_0\|_{L^{\infty}(Q)}^2+\|D\|_{L^{\infty}(Q)}^2+\|a_1\|_{L^{\infty}(Q)}^2)\int_t^s\|p_\epsilon(r)\|_{L^2(\mathcal{O})}^2\,dr\\
&+\|p_\epsilon(s)\|_{L^2(\mathcal{O})}^2.
\end{align*}
Denote by $\beta_2=\mathcal{L}_1\left(1+\|a_0\|_{L^{\infty}(Q)}+\|B_0\|_{L^{\infty}(Q)}^2+\|D\|_{L^{\infty}(Q)}^2+\|a_1\|_{L^{\infty}(Q)}^2\right),$ in view of Gronwall's inequality, yields
\begin{align*}
\|p_\epsilon(t)\|_{L^2(\mathcal{O})}^2\leq e^{\beta_2T}\|p_\epsilon(s)\|_{L^2(\mathcal{O})}^2
\end{align*}
for any $t\in [0,\frac{T}{2}]$ and any $T\geq s\geq t,$ which implies that
\begin{align}\label{eq4.12}
\|p_\epsilon(0)\|_{L^2(\mathcal{O})}^2+\int_0^{\frac{T}{2}}\|p_\epsilon(t)\|_{L^2(\mathcal{O})}^2\,dt
\leq &(1+\frac{2}{T})e^{\beta_2T}\int_{\frac{T}{4}}^{\frac{3T}{4}}\|p_\epsilon(t)\|_{L^2(\mathcal{O})}^2\,dt.
\end{align}
Thus, we conclude from inequalities \eqref{eq4.11}, \eqref{eq4.12} and the definition of $\tilde{\xi},$ $\tilde{\alpha}$ that
\begin{align}\label{eq4.13}
\nonumber\|p_\epsilon(0)\|_{L^2(\mathcal{O})}^2+\int_0^{\frac{T}{2}}\int_{\mathcal{O}}\tilde{\xi}^6|p_\epsilon(t)|^2e^{2s\tilde{\alpha}}\,dxdt
\leq &C(\lambda,s,T)\int_{\frac{T}{2}}^{\frac{3T}{4}}\xi^6\|p_\epsilon(t)\|_{L^2(\mathcal{O})}^2e^{2s\alpha}\,dt\\
\leq &C(\lambda, s,T)\int_{Q_\omega}\xi^7|v_\epsilon|^2e^{2s\alpha}\,dxdt.
\end{align}
On the other hand, it follows from inequality \eqref{eq4.11} and the definition of $\tilde{\xi},$ $\tilde{\alpha}$ that
\begin{align}\label{eq4.14}
\int_{\frac{T}{2}}^T\int_{\mathcal{O}}\tilde{\xi}^6|p_\epsilon|^2e^{2s\tilde{\alpha}}\,dxdt
\leq C\int_{Q_\omega}s\xi^7|v_\epsilon|^2e^{2s\alpha}\,dxdt.
\end{align}
Along with inequalities \eqref{eq4.13}-\eqref{eq4.14} and $\xi^7e^{2s\alpha}\leq C$ for some generic positive constant $C$ and any $(x,t)\in Q,$ we obtain
\begin{align}\label{eq4.15}
\|p_\epsilon(0)\|_{L^2(\mathcal{O})}^2+\int_Q\tilde{\xi}^6|p_\epsilon(t)|^2e^{2s\tilde{\alpha}}\,dxdt
\leq C(\lambda, s,T)\int_{Q_\omega}|v_\epsilon|^2\,dxdt.
\end{align}
Therefore, it follows from inequalities \eqref{eq4.10}, \eqref{eq4.15} and Young's inequality that
\begin{align}\label{eq4.16}
\frac{1}{\epsilon}\int_{\mathcal{O}}|y_\epsilon(T)|^2\,dx+\int_{Q_\omega}|v_\epsilon|^2\,dxdt
\leq C(s,\lambda,T)\left(\|\tilde{\xi}^{-3}e^{-s\tilde{\alpha}}g\|_{L^2(Q)}^2+\|y_0\|_{L^2(\mathcal{O})}^2\right).
\end{align}
We conclude from the regularity theory of fourth order linear parabolic equation and inequality \eqref{eq4.16} that
\begin{equation}\label{eq4.17}
\begin{cases}
\{v_\epsilon\}_{\epsilon>0}\,\,\,\textit{is\,\,\, uniformly\,\,\, bounded\,\,\, in}\,\,\, L^2(Q_\omega),\\
\{\frac{y_\epsilon(T)}{\sqrt{\epsilon}}\}_{\epsilon>0}\,\,\,\textit{is\,\,\, uniformly\,\,\, bounded\,\,\, in}\,\,\, L^2(\mathcal{O}),\\
\{y_\epsilon\}_{\epsilon>0}\,\,\,\textit{is\,\,\, uniformly\,\,\, bounded\,\,\, in}\,\,\, Y(Q),
\end{cases}
\end{equation}
which entails that there exists a subsequence (still denote by ${(y_\epsilon,v_\epsilon)}$) of ${(y_\epsilon,v_\epsilon)}$ and a function $(y,v)\in Y(Q)\times L^2(Q_\omega),$ such that 
\begin{align}\label{eq4.18}
(y_\epsilon,v_\epsilon)\rightharpoonup (y,v)\,\,\,\, \textit{in}\,\,\,\,Y(Q)\times L^2(Q_\omega)
\end{align}
and
\begin{align*}
y_\epsilon(T)\rightarrow y(T)\,\,\,\, \textit{in}\,\,\,\, L^2(\mathcal{O}).
\end{align*}
Hence, we pass to the limit in problem \eqref{eq4.8} and obtain that the pair $(y,v)$ is a solution to problem \eqref{eq4.7}. Moreover, we infer from inequality \eqref{eq4.16} and Fatou's theorem that
\begin{align}\label{eq4.19}
\int_{Q_\omega}|v|^2\,dxdt\leq C(s,\lambda,T)\left(\|\tilde{\xi}^{-3}e^{-s\tilde{\alpha}}g\|_{L^2(Q)}^2+\|y_0\|_{L^2(\mathcal{O})}^2\right)
\end{align}
and 
\begin{align}\label{eq4.20}
y(x,T)=0,\,\,\,\forall\,\,x\in\mathcal{O}.
\end{align}
i.e., there exists a $v\in L^2(Q_\omega)$ such that the solution $y$ of problem \eqref{eq4.1} with $y(x,T)=0$ for any $x\in\mathcal{O}.$

From the fact that $(y_\epsilon,v_\epsilon)$ is the solution of to the extremal problem \eqref{eq4.5}-\eqref{eq4.7}, we deduce that for any admissible control $(w,u)\in \mathcal{V}$ with $w(T)\equiv 0,$
\begin{align*}
\mathcal{J}_\epsilon(y_\epsilon,v_\epsilon)\leq \mathcal{J}_\epsilon(w,u)=\frac{1}{2}\int_{Q_\omega}|u|^2\,dxdt,
\end{align*}
 which implies that
\begin{align*}
\int_{Q_\omega}|v_\epsilon|^2\,dxdt\leq \int_{Q_\omega}|u|^2\,dxdt
\end{align*}
and
\begin{align}\label{eq4.21}
\limsup_{\epsilon\rightarrow 0}\int_{Q_\omega}|v_\epsilon|^2\,dxdt\leq \int_{Q_\omega}|v|^2\,dxdt.
\end{align}
It follows from \eqref{eq4.18} and Fatou's theorem that
\begin{align}\label{eq4.22}
\int_{Q_\omega}|v|^2\,dxdt\leq\liminf_{\epsilon\rightarrow 0}\int_{Q_\omega}|v_\epsilon|^2\,dxdt\leq \int_{Q_\omega}|u|^2\,dxdt
\end{align}
for any admissible control $(w,u)\in \mathcal{V}$ with $w(T)=0,$ i.e., $v$ is a control of minimal norm in $L^2(Q_\omega)$ among the admissible controls (such that \eqref{eq4.2} is satisfied) and $v_\epsilon\rightarrow v$ in $L^2(Q_\omega).$
\end{proof}
In what follows, we will prove the null controllability of problem \eqref{eq1.1} by the Leray-Schauder's fixed point theorem.
\begin{theorem}\label{th4.2}
Let the functions $\tilde{\xi},$ $\tilde{\alpha}$ be defined by \eqref{eq4.3} and \eqref{eq4.4}, respectively. If $\tilde{\xi}^{-3}e^{-s\tilde{\alpha}}g\in L^2(Q)$ and $y_0\in H_0^1(\mathcal{O})\cap H^2(\mathcal{O}).$ Then there exists at least one control $v\in L^2(Q_\omega)$ such that the solution $y$ of problem \eqref{eq1.1} satisfies $y(x,T)\equiv 0$ for any $x\in\mathcal{O}.$ 
\end{theorem}

\begin{proof}
Let $z\in L^2(0,T;H_0^1(D)\cap H^2(D))$ be given, consider the following problem 
\begin{equation}\label{eq4.23}
\begin{cases}
\frac{\partial y}{\partial t}+\Delta^2y+a_0y+B_0\cdot\nabla y+D:\nabla^2y+a_1\Delta y=G_1(z,\nabla z,\nabla^2z)y+G_2(z,\nabla z,\nabla^2z)\cdot \nabla y\\
+G_3(z,\nabla z,\nabla^2z): \nabla^2 y+F(0,0,0)+v\chi_\omega+g,\,\,\,\,(x,t)\in Q,\\
y=\Delta y=0,\,\,(x,t)\in\Sigma,\\
y(x,0)=0,\,\,\,\,\,x\in \mathcal{O},
\end{cases}
\end{equation}
where
\begin{align*}
&G_1(w,\nabla w,\nabla^2 w)=\int_0^1\frac{\partial F}{\partial y}(\tau w,\tau\nabla w,\tau\nabla^2w)\,d\tau,\\
&G_2(w,\nabla w,\nabla^2w)=\int_0^1\nabla_pF(\tau w,\tau\nabla w,\tau\nabla^2w)\,d\tau,\\
&G_3^{ij}(w,\nabla w,\nabla^2w)=\int_0^1\frac{\partial F}{\partial r_{ij}}(\tau w,\tau\nabla w,\tau\nabla^2w)\,d\tau,\\
&r_{ij}=\frac{\partial^2y}{\partial x_i\partial x_j},\,\,\,1\leq i,j\leq n.
\end{align*}
Since $F\in W^{1,\infty}(\mathbb{R}\times\mathbb{R}^n\times\mathbb{R}^{n^2},\mathbb{R}),$ there exists a positive constant $M,$ such that
\begin{align*}
|G_1(u,p,r)|+|G_2(u,p,r)|+|G_3(u,p,r)|\leq M,\,\,\,\,\,\forall\,\,\,\,(u,p,r)\in\mathbb{R}\times\mathbb{R}^n\times\mathbb{R}^{n^2}
\end{align*}
and
\begin{align*}
|F_y(u,p,r)|+|\nabla_pF(u,p,r)|+\sum_{i,j=1}^n\left|\frac{\partial F}{\partial r_{ij}}(u,p,r)\right|\leq M,\,\,\,\,\,\forall\,\,\,\,(u,p,r)\in\mathbb{R}\times\mathbb{R}^n\times\mathbb{R}^{n^2}.
\end{align*}

From Theorem \ref{th4.1}, we conclude that there exists at least one control $v\in L^2(Q_\omega),$ such that the solution $y^z$ of problem \eqref{eq4.23} satisfies
\begin{align}\label{eq4.24}
y^z(x,T)\equiv 0,\,\,\forall\,\,x\in \mathcal{O}.
\end{align}
Moreover, we also have
\begin{align}\label{eq4.25}
\|v\|_{L^2(Q_\omega)}\leq C(s,\lambda,T)\left(\|\tilde{\xi}^{-3}e^{-s\tilde{\alpha}}g\|_{L^2(Q)}^2+\|y_0\|_{L^2(\mathcal{O})}^2\right).
\end{align}
In what follows, we denote by $v^z$ the control with the minimal $L^2(Q_\omega)$-norm in the set of 
the controls such that the solution $y^z$ of problem \eqref{eq4.23} corresponding to $z$ satisfies $y^z(T)\equiv 0.$

From the regularity theory of parabolic equations, we conclude that there exists a unique weak solution $y^z\in X=L^2(0,T;H_0^1(\mathcal{O})\cap H^4(\mathcal{O}))\cap H^1(0,T;L^2(\mathcal{O})).$  Moreover, since $F\in W^{1,\infty}(\mathbb{R}\times\mathbb{R}^n\times\mathbb{R}^{n^2};\mathbb{R}),$  there exists a positive constant $C$ independent of $z,$ such that
\begin{align}\label{eq4.26}
\nonumber\|y^z\|_X\leq  &C(\|y_0\|_{H^2(\mathcal{O})}+\|F(0,0,0)+v^z\chi_\omega+g\|_{L^2(Q)})\\
\leq  &C(1+\|y_0\|_{H^2(\mathcal{O})}+\|v^z\|_{L^2(Q_\omega)}+\|g\|_{L^2(Q)}).
\end{align}
Thus, along with inequalities \eqref{eq4.25}-\eqref{eq4.26}, we deduce that there exists a positive constant $\mathcal{L}_1$ independent of $z,$ such that 
\begin{align}\label{eq4.27}
\|y^z\|_X\leq  \mathcal{L}_1\left(1+\|y_0\|_{H^2(\mathcal{O})}+\|\tilde{\xi}^{-3}e^{-s\tilde{\alpha}}g\|_{L^2(Q)}\right).
\end{align}
Define $\Lambda:L^2(0,T;H_0^1(\mathcal{O})\cap H^2(\mathcal{O}))\rightarrow L^2(0,T;H_0^1(\mathcal{O})\cap H^2(\mathcal{O}))$ by
\begin{align*}
\Lambda(z)=y^z,
\end{align*}
then the mapping $\Lambda$ is well-defined. In what follows, we will prove the existence of a fixed point for the operator $\Lambda$ by the Leray-Schauder's fixed points Theorem. To this purpose, we will first prove that $\Lambda$ is continuous, i.e., if $z_k\rightarrow z$ in $L^2(0,T;H_0^1(\mathcal{O})\cap H^2(\mathcal{O})),$ we have $\Lambda(z_k)\rightarrow \Lambda(z).$

Denote by $y^k=\Lambda(z_k),$ where $(y^k,q^k)$ is the solution of problem
\begin{equation}\label{eq4.28}
\begin{cases}
\frac{\partial y^k}{\partial t}+\Delta^2y^k+a_0y^k+B_0\cdot\nabla y^k+B:\nabla^2y^k+a_1\Delta y^k=G_1(z_k,\nabla z_k,\nabla^2 z_k)y^k\\
+G_2(z_k,\nabla z_k,\nabla^2 z_k)\cdot \nabla y^k+G_3(z_k,\nabla z_k,\nabla^2 z_k): \nabla^2 y^k+F(0,0,0)+v_{z_k}\chi_\omega+g,\,\,\,\,(x,t)\in Q,\\
y^k=\Delta y^k=0,\,\,\,\,\,(x,t)\in\Sigma,\\
y^k(x,0)=0,\,\,\,\,\,x\in\mathcal{O}.
\end{cases}
\end{equation}
It follows from inequality \eqref{eq4.27} and the fact that $z_k\rightarrow z$ in $L^2(0,T;H_0^1(\mathcal{O})\cap H^2(\mathcal{O}))$ that
\begin{align*}
&\{y^k\}_{k=1}^{\infty}\,\,\,\textit{is\,\,\,uniformly\,\,\,bounded\,\,in}\,\,X,\\
&\{v^{z_k}\}_{k=1}^{\infty}\,\,\,\textit{is\,\,\,uniformly\,\,\,bounded\,\,in}\,\,L^2(Q_\omega),
\end{align*}
which entails that there exists a subsequence of $\{y^k\}_{k=1}^\infty,$ $\{v^{z_k}\}_{k=1}^\infty$ (still denote by themselves) and $y\in X,$ $v\in L^2(Q_\omega),$ such that
\begin{align*}
&y^k\rightharpoonup y\,\,\,\textit{in}\,\,X\,\,\,\textit{as}\,\,k\rightarrow+\infty,\\
&y^k\rightarrow y\,\,\,\textit{in}\,\,L^2(0,T;H_0^1(\mathcal{O})\cap H^2(\mathcal{O}))\,\,\,\textit{as}\,\,k\rightarrow+\infty,\\
&v^{z_k}\rightharpoonup v\,\,\,\textit{in}\,\,L^2(Q_\omega)\,\,\,\textit{as}\,\,k\rightarrow+\infty.
\end{align*}
Since $F\in W^{1,\infty}(\mathbb{R}\times\mathbb{R}^n\times\mathbb{R}^{n^2},\mathbb{R}),$ we conclude that there exists a subsequence of $\{G_1(z_k,\nabla z_k,\nabla^2 z_k)\}_{k=1}^\infty,$ $\{G_2(z_k,\nabla z_k,\nabla^2 z_k)\}_{k=1}^\infty,$ $\{G_3(z_k,\nabla z_k,\nabla^2 z_k)\}_{k=1}^\infty,$ (still denote by themselves), such that
\begin{align*}
&G_1(z_k,\nabla z_k,\nabla^2z_k)\rightarrow G_1(z,\nabla z,\nabla^2 z)\,\,\,\textit{weakly\,\,star\,\,in}\,\,L^{\infty}(Q),\,\,\textit{as}\,\,k\rightarrow+\infty,\\
&G_2(z_k,\nabla z_k,\nabla^2z_k)\rightarrow G_2(z,\nabla z,\nabla^2z)\,\,\,\textit{weakly\,\,star\,\,in}\,\,L^{\infty}(Q),\,\,\textit{as}\,\,k\rightarrow+\infty,\\
&G_3(z_k,\nabla z_k,\nabla^2z_k)\rightarrow G_3(z,\nabla z,\nabla^2z)\,\,\,\textit{weakly\,\,star\,\,in}\,\,L^{\infty}(Q),\,\,\textit{as}\,\,k\rightarrow+\infty.
\end{align*}
Let $k\rightarrow+\infty$ in problem \eqref{eq4.28}, we obtain
\begin{equation}\label{eq4.29}
\begin{cases}
\frac{\partial y}{\partial t}+\Delta^2y+a_0y+B_0\cdot\nabla y+B:\nabla^2y+a_1\Delta y=G_1(z,\nabla z,\nabla^2 z)y+G_2(z,\nabla z,\nabla^2z)\cdot \nabla y\\
+G_3(z,\nabla z,\nabla^2z): \nabla^2 y+F(0,0,0)+v\chi_\omega+f,\,\,\,\,(x,t)\in Q,\\
y=\Delta y=0,\,\,\,\,\,(x,t)\in\Sigma,\\
y(x,0)=y_0(x),\,\,\,\,\,x\in \mathcal{O}
\end{cases}
\end{equation}
and
\begin{align}\label{eq4.30}
y(x,T)\equiv 0,\,\,\forall\,\,x\in \mathcal{O},
\end{align}
which entails that $y=\Lambda(z).$ Thus, we have proved that $\Lambda(z_k)\rightarrow\Lambda(z)$ in $L^2(0,T;H_0^1(\mathcal{O})\cap H^2(\mathcal{O})),$ i.e., the mapping $\Lambda:L^2(0,T;H_0^1(\mathcal{O})\cap H^2(\mathcal{O}))\rightarrow L^2(0,T;H_0^1(\mathcal{O})\cap H^2(\mathcal{O}))$ is continuous. Thanks to the compactness of $X\subset L^2(0,T;H_0^1(\mathcal{O})\cap H^2(\mathcal{O}))$ and inequality \eqref{eq4.27}, we conclude that the mapping $\Lambda:L^2(0,T;H_0^1(\mathcal{O})\cap H^2(\mathcal{O})))\rightarrow L^2(0,T;H_0^1(\mathcal{O})\cap H^2(\mathcal{O}))$ is compact. 
Denote by 
\begin{align*}
\mathcal{R}_1=\mathcal{L}_1\left(1+\|y_0\|_{H^2(\mathcal{O})}+\|\tilde{\xi}^{-3}e^{-s\tilde{\alpha}}g\|_{L^2(Q)}\right)
\end{align*}
and
\begin{align*}
B=\{u\in L^2(0,T;H_0^1(\mathcal{O})\cap H^2(\mathcal{O})):\|u\|_{L^2(0,T;H_0^1(\Omega)\cap H^2(\mathcal{O}))}\leq\mathcal{R}_1\},
\end{align*}
then $\Lambda: B\rightarrow B.$ Thus, we can employ the Leray-Schauder's fixed points Theorem to conclude that the operator $\Lambda$ possesses at least one fixed point $y\in L^2(0,T;H_0^1(\Omega)\cap H^2(\mathcal{O})).$ That is, for any $y_0\in H^2(\mathcal{O})\cap H_0^1(\mathcal{O}),$ there exist at least one control $v\in L^2(Q_\omega),$ such that the corresponding solutions to problem \eqref{eq1.1} satisfy $y(x,T)\equiv 0$ for any $x\in \mathcal{O}.$
\end{proof}

\begin{remark}
Under the same assumptions as in Theorem \ref{th4.2}. If the function $F(y,\nabla y,\nabla^2y)$ is independent of $\nabla y$ and $\nabla^2y,$ i.e., $F(y,\nabla y,\nabla^2y)\equiv G(y)$ for some locally Lipschitz continuous function $G,$  then for any $y_0\in L^2(\mathcal{O}),$ the same conclusion as in Theorem \ref{th4.2} remains true.
\end{remark}

\section*{Acknowledgement}
Partial financial support was received from the National Science Foundation of China Grant (11801427, 11871389) and the Fundamental Research Funds for the Central Universities (xzy012022008, JB210714).

\bibliographystyle{abbrv}
\bibliography{BIB}
\end{document}